\documentclass[11pt]{article}
\usepackage{amsthm, amsmath , amsfonts,  eucal } 

\usepackage{graphics}
\makeatletter
\renewcommand{\@evenhead}{\it\thepage\hfill 
Fass\`o and Sasonetto }
\renewcommand{\@oddhead}{\it 
  Fass\`o and Sansonetto:
  Nearly-integrable almost-symplectic Hamiltonian systems 
\hfill \thepage}
\makeatother

\theoremstyle{plain}
\newtheorem{theorem}{\bf Theorem}
\newtheorem{proposition}[theorem]{\bf Proposition}
\newtheorem{lemma}[theorem]{\bf Lemma}
\newtheorem{definition}[theorem]{\bf Definition}

\usepackage{titlesec}
\titleformat{\section}
  {\normalfont\large\bfseries}{\thesection}{0.5em}{}
\newcommand{\addperiod}[1]{#1.}
\titleformat{\subsection}[runin]
  {\normalfont\bfseries}{\thesubsection}{0.25em}{\addperiod}

\begin{document}

%%%=================MACROS====================================

%%%-----------------Lists
\newcommand\bList{
\begin{list}{}{\leftmargin2em\labelwidth1.2em\labelsep.5em\itemindent0em
\topsep0ex\itemsep-.8ex} }
\newcommand\eList{\end{list}}

%%%-----------------Displyed_equations_and_related

\newcommand{\beq}[1]{\begin{equation}\label{#1}}
\newcommand{\eeq}{\end{equation}}
\newcommand{\for}[1]{(\ref{#1})}
\newcommand{\ug}{\;=\;}

%%%-----------------Symbols
\renewcommand{\a}{\alpha}
\newcommand{\e}{\varepsilon}
\newcommand{\p}{\varphi}
\renewcommand{\P}{\Phi}
\newcommand{\s}{\sigma}
\renewcommand{\o}{\omega}
\renewcommand{\L}{\Lambda}

\newcommand{\cA}{\mathcal{A}}                 
\newcommand{\cB}{\mathcal{B}}
\newcommand{\cC}{\mathcal{C}}                 
\newcommand{\cF}{\mathcal{F}}                 
\newcommand{\cK}{\mathcal{K}}                 
\newcommand{\cL}{\mathcal{L}}                 
\newcommand{\cO}{\mathcal{O}}                 
\newcommand{\cS}{\mathcal{S}}                 

\newcommand{\Sp}[1]{ \mathrm{Sp}(#1)}

\newcommand{\toro}{\mathbb{T}}
\newcommand{\reali}{\mathbb{R}}
\newcommand{\razionali}{\mathbb{Q}}
\newcommand{\interi}{\mathbb{Z}}
\newcommand{\tenne}{\toro^n}
\newcommand{\renne}{\reali^n}

\renewcommand{\bar}{\overline}		
\newcommand{\der}[2]{\frac{\partial#1}{\partial#2}}
\newcommand{\dder}[3]{\frac{\partial^2#1}{\partial#2\partial#3}}
\newcommand{\const}{\mathrm{const}} 
\newcommand\rank{\mathrm{rank\,}}

\renewcommand\aa{\mathrm {aa}}

%%%=================TITLE_ETC====================================

\title{\bf \parskip=0pt
Nearly-integrable almost-symplectic
\\
Hamiltonian systems\footnote{
This work is part of the research projects {\it
Symmetries and integrability of nonholonomic
mechanical systems} of the University of Padova.}}

\author{
Francesco Fass\`o\footnote{
Universit\`a di Padova,
Dipartimento di Matematica,
Via Trieste 63, 35121 Padova, Italy.
{\tt (E-mail: fasso@math.unipd.it)}
}
\ and 
N. Sansonetto\footnote{
Universit\`a di Padova,
Dipartimento di Matematica,
Via Trieste 63, 35121 Padova, Italy.
{\tt (E-mail: nicola.sansonetto@gmail.com)} }
\ \footnote
{Supported by the Research Project
{\it Symmetries and integrability of nonholonomic
mechanical systems} of the University of Padova.}
}

\vskip 1truecm
\date{\small (\today)}

\maketitle

{\small
\begin{abstract}
\noindent
Integrable Hamiltonian systems on almost-symplectic manifolds have
recently drawn some attention. Under suitable properties, they have a
structure analogous to those of standard symplectic-Hamiltonian
completely integrable systems.  Here we study small Hamiltonian
perturbations of these systems.  Preliminarily, we investigate some
general properties of these systems. In particular, we show that if
the perturbation is `strongly Hamiltonian' (namely, its Hamiltonian
vector field is also a symmetry of the almost-Hamiltonian structure)
then the system reduces, under an almost-symplectic version of
symplectic reduction, to a family of nearly integrable  standard
symplectic-Hamiltonian  vector fields on a reduced phase space, of
codimension not less than 3. Therefore, we restrict our study to
non-strongly Hamiltonian perturbations. We will show that KAM
theorem on the survival of strongly nonresonant quasi-periodic tori
does non apply, but that a weak version of Nekhoroshev theorem on the
stability of actions is instead valid, even though for a time scale
which is polynomial (rather than exponential) in the inverse of the
perturbation parameter. 
\noindent
\vskip 1truecm
\end{abstract}
}

{\small

}

%%%============ARTICLE

\section{Introduction}
An almost-symplectic manifold is a generalization of a symplectic
manifold, in which the nondegenerate 2-form is not  closed.
Hamiltonian systems on almost-symplectic manifolds arise, for
instance, in nonholonomic mechanics \cite{bates-sn}; moreover, their
study might be an intermediate step to the study of the more general
case of (generalized) Hamiltonian systems on (twisted) Dirac manifolds
considered in \cite{BlvdS,SW,zung1}. The main difference between Hamiltonian
systems on symplectic and almost-symplectic manifolds is that in the
almost-symplectic case, due to the non-closedness of the 2-form,
Hamiltonian vector fields are not necessarily symmetries of the
almost-symplectic structure. Those which are symmetries of the
almost-symplectic structure have special properties, and resemble
more closely the Hamiltonian vector fields of the standard symplectic
case; they were called `strongly Hamiltonian' in~\cite{FS1}.

We are aware of only a few articles dedicated to the almost-symplectic
case. Our previous work \cite{FS1} focussed on the integrability of
Hamiltonian systems on almost-symplectic manifolds. Reference
\cite{vaisman}, by I. Vaisman, studies general properties of strongly
Hamiltonian systems on almost-symplectic manifolds and provides
examples. Some geometric aspects are studied in \cite{SS1,sjamaar}.

The question underlying the present work is how different is the
{dynamics} of an almost-symplectic Hamiltonian system from the
standard symplectic-Hamiltonian ones. This is a broad question, which
probably does not have a single, definite answer. For instance, in
the almost-symplectic case, non-strongly Hamiltonian systems need not
conserve the volume in phase space, while strongly Hamiltonian
systems do (see Section 2). In this paper, we begin this
investigation by considering a special case, that  of
almost-symplectic Hamiltonian systems that are small perturbations of
almost-symplectic integrable Hamiltonian systems. The question we ask
is whether the great theorems of Hamiltonian perturbation
theory---KAM \cite{AKN} and Nekhoroshev \cite{nek77,bengall,poeschel}
theorems---retain their validity in the almost-symplectic framework.
Not surprisingly, the answer depends crucially on whether the
perturbation is assumed to be Hamiltonian or strongly Hamiltonian. 

Preliminarily to this study, we need to further investigate some
properties of nearly integrable Hamiltonian systems on
almost-symplectic manifolds. We will do this in Sections 2 and 3.
The first question we investigate are the properties---and in a way
the very existence---of  strongly Hamiltonian perturbations. We will
investigate this question under certain hypotheses of genericity on
the almost-symplectic structure and on the perturbation and show
that, at least under such hypotheses, strongly Hamiltonian
nearly-integrable almost-symplectic Hamiltonian systems are in some
way not deeply different from standard symplectic-Hamiltonian
systems. Specifically, each of them reduce, under an analog of the
standard Meyer-Marsden-Weinstein symplectic reduction that was
studied in \cite{vaisman}, to a family of nearly integrable standard
symplectic-Hamiltonian systems. Versions of KAM and Nekhoroshev
theorem for strongly Hamiltonian perturbations could be easily
obtained, but have limited novelty.

More interestingly, we will see that if the perturbation is
Hamiltonian, but not strongly Hamiltonian, then KAM theorem does not
apply: quasi-periodic motions do not survive small perturbations.
However, a weaker version of Nekhoroshev theorem that gives {\it
stability of all motions for polynomial times} does hold.
Specifically, we will prove that, if the unperturbed system has the
standard  (quasi) convexity property of Nekhoroshev theory, then the
variations of the actions in all motions is bounded by quantities of
order $\e^{c_1}$, with a positive constant~$c_1$, over times
$\e^{-c_2}$, with a constant $c_2>1$,  where $\e$ is the size of the
perturbation. This time scale is much shorter than Nekhoroshev's
stability time scale for symplectic-Hamiltonian systems, which is
{exponential} in $1/\e^\const$.

The basic ideas of perturbation theory in the almost-symplectic
context will be explained in Section 4, in a form that should make
plausible (and self-evident for the reader expert in Hamiltonian
perturbation theory) our almost-symplectic version of Nekhoroshev
theorem. Since the proof of this result is almost identical to that
of the symplectic case (as it can be found, e.g., in \cite{poeschel})
we will not reproduce it here.

\section{Hamiltonian systems on almost-symplectic
manifolds}

\subsection{Hamiltonian and strongly Hamiltonian vector fields}
We consider a connected manifold $M$ of even dimension $2n$ equipped
with a nondegenerate 2-form~$\s$. $(M,\s)$ is a symplectic manifold
if $\s$ is closed. We will say that $(M,\s)$ is an {\it
almost-symplectic manifold} if $\s$ is not closed. The nondegeneracy
requires $n\ge 2$.

Following \cite{FS1}, we say that a vector field $X$ on an
almost-symplectic manifold $(M,\s)$ is 
\bList
\item[i.] {\it Hamiltonian} if $i_X\sigma$ is exact. Thus $i_X\sigma=
-df$ for some function $f\in C^{\infty}(M)$ that we call a
Hamiltonian of $X$, and we will write $X_f$ for $X$.
\item[ii.] {\it Strongly Hamiltonian} if it is
Hamiltonian and, moreover, it is a symmetry of $\sigma$, that is,
$$
   L_X\sigma = 0 \,,
$$
where $L$ denotes the Lie derivative. By Cartan's magic formula
$L_X\s=i_X(d\s)+d(i_X\s)$, the strong Hamiltonianity of a
Hamiltonian vector field is equivalent to
\begin{equation}
\label{ixdsigma}
   i_Xd\sigma= 0 \,.
\end{equation}
\eList
In both cases, we will call $n$ the number of degrees of freedom of
the system. Moreover, following partly \cite{vaisman}, we say that 
\bList
\item[iii.] A {\it strongly Hamiltonian function} is any function
on $M$ whose Hamiltonian vector field is strongly Hamiltonian.
We will denote by $S^\infty(M)$ the subset of $C^\infty(M)$ consisting
of strongly Hamiltonian functions.
\eList

\vspace{3ex}\noindent
{\it Remarks: } (i) Reference \cite{vaisman} considered only
the case of strongly Hamiltonian vector fields, and called
them Hamiltonian. We adhere here to the terminology that we
used in \cite{FS1} because we will consider both classes of
Hamiltonian and strongly Hamiltonian vector fields and need to
distinguish among them.

(ii) Most of the following could be generalized, in analogy with the
standard symplectic case, to `local' Hamiltonian and strongly
Hamiltonian vector fields, as done in \cite{vaisman}. We do not
consider such a greater generality because we will not need it in the
study of integrable systems and their perturbations.

\vspace{3ex}\noindent
At an algebraic level, the reason for considering strongly Hamiltonian
vector fields is the following. The almost-symplectic form $\sigma$
induces an almost-Poisson bracket on smooth functions of $M$, which
is defined by
\begin{equation}
\label{eq:aP}
   \{f,g\} := - \sigma(X_f,X_g) \qquad \forall\,f,g\in C^\infty(M) \,.
\end{equation}
Because of the nonclosedness of $\sigma$, this bracket does not
satisfy the Jacobi identity. Therefore, it does not make
$C^\infty(M)$ a Lie algebra and does not induce a
\hbox{(anti-)}homomorphism between functions and Hamiltonian vector
fields on $M$. However, all this holds true for the restriction of
the bracket to $S^\infty(M)$. This is a consequence of the following
Lemma (from \cite{FS1}, where its statement contains however an
obvious flaw):

\begin{lemma}
\label{l:homomorphism} 
Let $Y$ and $Z$ be two vector fields on $(M,\sigma)$. If $Y$
is Hamiltonian and $Z$ is a symmetry of $\s$ then $[Y,Z]$ is
Hamiltonian with Hamiltonian $-\sigma(Y,Z)$:
$$
    [Y,Z] = - X_{\sigma(Y,Z)} \,.
$$
\end{lemma}

\begin{proof} Since $d(i_Y\s)=0$ and $L_Z\s=0$,
$ d(\sigma(Y,Z)) = d(i_Zi_Y\s) =
L_Z(i_Y\sigma) - i_Z\, d(i_Y\sigma) = i_Y(L_Z\sigma)+i_{[Z,Y]}\sigma
= i_{[Z,Y]}\sigma = - i_{[Y,Z]}\sigma$.
\end{proof}

Applied to strongly Hamiltonian vector fields, Lemma
\ref{l:homomorphism} gives
$$
  [X_f,X_g] = - X_{\{f,g\}} \qquad \forall \, f,g\in S^\infty(M) \,.
$$
Since the Lie bracket of two symmetries of $\s$ is still a symmetry
of $\s$, this shows that the set of strongly Hamiltonian vector
fields is a Lie subalgebra of the algebra of vector fields on $M$.
Correspondingly, $S^\infty(M)$ is a Lie algebra when equipped with
the bracket (\ref{eq:aP}) and $f\mapsto X_f$ is an anti-homomorphism
between these two Lie algebras.

As pointed out in \cite{vaisman}, in view of (\ref{ixdsigma}) and of
Lemma \ref{l:homomorphism}, strongly Hamiltonian vector fields
form a distribution $\tilde\cS$ on $M$. This distribution is a
subdistribution of the distribution $\cK_{d\s}$ whose fiber is, at
each point, the kernel of $d\sigma$ at that point.\footnote{
We recall that the kernel of a $3$-form $\eta$ at a point $m\in M$ is
the kernel of the linear map $T_mM \to \Lambda^2(T_mM)$ given by
contraction with $\eta_m$, namely the map $v\mapsto i_v\eta_m$, $v\in
T_mM$. Here,  $\Lambda^2(T_mM)$ denotes the space of all covariant
antisymmetric 2-tensors on $T_mM$.}
In the present case, given that $d\s$ is closed,
$\cK_{d\s}$ is integrable and coincides with the so called
characteristic distribution of $d\s$ \cite{godbillon,LibMarle}. 
Some properties of the distribution $\tilde\cS$ have been studied in
\cite{vaisman}.

Reference \cite{vaisman} remarks that the class of strongly
Hamiltonian vector fields (which may be identified, modulo constants,
with the class $S^\infty(M)$ of strongly Hamiltonian functions) 
might be much smaller than that of Hamiltonian vector fields (which
may be identified, modulo constants, with $C^\infty(M)$), and that it
is not even apriori clear whether $S^\infty(M)$ contains any
non-constant function. In order to show that this is not the case,
reference \cite{vaisman} provided some examples. We add here a
simple, quantitative remark on this question, that we will use in the
sequel:

\begin{lemma}
\label{l:upperbound}
At any point at which $d\s$ is nonzero there exist at most $2n-3$
germs of functionally independent strongly Hamiltonian functions.

Equivalently: if $d\s(m)\not=0$ at some $m\in M$, there exists a
coordinate system in a neighbourhood $V$ of $m$ such that the
restriction to $V$ of any strongly Hamiltonian function does not
depend on three coordinates.
\end{lemma}

\begin{proof} We use the algebraic fact that {\it at a point at which
a 3-form is nonzero, the codimension of its kernel is $\ge 3$}. This
must of course be known, but since we could not find a reference we
provide a proof.  Let $\eta$ be a 3-form on $M$ and $\eta_m\not=0$ at
some $m\in M$. Then there exists a vector $v\in T_mM\setminus\{0\}$
such that the 2-form $i_v\eta_m\not=0$. Since any nonzero 2-form,
being antisymmetric, has positive even rank, this implies $\rank
i_v\eta_m\ge 2$. The conclusion now follows observing that $\ker
\eta_m\subset\ker i_v\eta_m$ given that the latter contains $v$ while
the former does not; hence $\dim(\ker i_v\eta_m\setminus
\ker\eta_m)\ge 1$ and the conclusion follows.

Assume now $d\s\not=0$ at a point $m$. Then, $d\s$ is everywhere
nonzero in any sufficiently small neighbourhood of $m$. If $V_0$ is
one such neighbourhood, then in $V_0$ the leaves of $\cK_{d\s}$ have
dimension $\le 2n-3$ and any set of  sections of $\cK_{d\s}$ that are
linearly independent at each point of $V_0$ has cardinality $\le
2n-3$. In turn, by the nondegeneracy of $\s$, there are at most
$2n-3$  strongly Hamiltonian functions which are defined in a
neighbourhood $V\subseteq V_0$ of $m$ and are everywhere functionally
independent. The statement in terms of coordinates follows by
restricting $V$ if necessary and completing a set of functionally
independent strongly Hamiltonian functions to a coordinate system.
\end{proof}

The upper bound of Lemma \ref{l:upperbound} is {\it de facto} met in
all examples in \cite{vaisman}.

Dynamically, strongly Hamiltonian vector fields have special
properties among the class of Hamiltonian vector fields.
For instance, Hamiltonian vector fields need not preserve the volume
$\s^n$. An example is the vector field $X=-\der{}{x_3}$ on
$M=\reali^4\setminus\{0\}$ with almost-symplectic form
$\s=x_3dx_1\wedge dx_2+dx_3\wedge dx_4$. Instead, since
$L_X\s^n=\s^{n-1}\wedge L_X\s$, we have the following

\begin{proposition}
\label{p:ConsVol} Every strongly Hamiltonian vector field on an
almost-symplectic manifold $(M,\s)$ preserves the volume $\s^n$.
\end{proposition}

\subsection{Strongly Hamiltonian completely integrable systems}
\label{ss:SHCIS}
As shown in \cite{FS1}, the well known notion of complete integrability
of the symplectic case and the resulting structure described by the
Liouville-Arnold theorem, are a particular case of a more general
situation, that holds in the almost-symplectic case. We begin
recalling the following almost-symplectic version of the
Liouville-Arnold theorem:

\begin{proposition}
\label{p:CI} {\rm \cite{FS1} }
Let $(M,\s)$ be an almost-symplectic manifold of dimension $2n$ and
$\pi = (\pi_1,\ldots,\pi_n) : M \to \reali^{n}$ a submersion with
compact and connected fibers whose components $\pi_1,\ldots,\pi_n$
are 
\bList
\item[$\bullet$] strongly Hamiltonian functions 
\item[$\bullet$] pairwise in involution with respect to the
almost-Poisson bracket \for{eq:aP}, namely
$\{\pi_i,\pi_j\}=0$ for $i,j=1,\ldots,n$.
\eList
Then the fibers of $\pi$ are diffeomorphic to $\tenne$ and each of
them has a neighbourhood $V$ equipped with coordinates $(a,\alpha):
V\to\cA\times\tenne$, with $\cA\subseteq\renne$, such that
$\pi=\pi(a)$ and the local representative $\s_\aa$ of $\s$ in these
coordinates has the form
\begin{equation}\label{eq:sigma}
  \s_\aa = \sum_{i=1}^n da_i\wedge d\alpha_i
  + \frac12 \sum_{i,j=1}^n A_{ij}(a) da_i\wedge da_j
\end{equation}
where $A$ is an $n\times n$ antisymmetric matrix that depends
smoothly on $a$.
\end{proposition}

The hypothesis that $\pi_1,\ldots,\pi_n$ are strongly Hamiltonian
functions is essential for this result to hold. It ensures that the
Hamiltonian vector fields of these functions, besides being tangent
to the fibers of $\pi$ on account of the involutivity hypothesis
($L_{X_{\pi_i}}\pi_j=-\s(X_{\pi_i},X_{\pi_j})=-\{\pi_i,\pi_j\}=0$),
do pairwise commute ($[X_{\pi_i},X_{\pi_j}]=-X_{\{\pi_i,\pi_j\}}=0$)
and thus give the fibers of $\pi$ the structure of the
$n$-dimensional torus; for details and comments see \cite{FS1}.

The coordinates $(a,\alpha)$ will be called action-angle coordinates
relative to $\pi$. From (\ref{eq:sigma}) it follows that, in these
coordinates, the Hamiltonian vector field $X_f=\sum_{i=1}^n
(X_f^{a_i}\partial_{a_i} + X^{\alpha_i}_f\partial_{\alpha_i})$ of a
function $f(a,\alpha)$ has components
\begin{equation}\label{eq:HVF}
  X_f^a = - \der f {\alpha} \,,\qquad
  X_f^\alpha  = \der fa + A\der f \alpha 
\end{equation}
and that the almost-Poisson brackets (\ref{eq:aP}) have the expression
\begin{equation}\label{eq:APB}
 \{f,g\}_\aa = 
  \sum_{i=1}^n\Big( \der f{a_i} \der g{\alpha_i} - 
                    \der f{\alpha_i} \der g{a_i} \Big)
  + 
  \sum_{i,j=1}^n A_{ij}\der f{\alpha_i} \der g{\alpha_j} \,. 
\end{equation}
Moreover,
\begin{equation}\label{eq:dsigma}
  d\s_\aa \;=\; \frac12
  \sum_{i,j,k=1}^n \der{A_{ij}}{a_k} da_k\wedge da_i\wedge da_j
  \;=\; \sum_{i,j,k=1}^n C_{ijk} da_k\otimes da_i\otimes da_j 
\end{equation}
where
\begin{equation}\label{eq:C}
  C_{ijk}(a) = \der{A_{ij}}{a_k}(a) + \der{A_{ki}}{a_j}(a) 
  + \der{A_{jk}}{a_i}(a) \,. 
\end{equation}
The condition that $\s_\aa$ is not symplectic is precisely that 
the skew-symmetric 3-tensor field $C$ with components $C_{ijk}$ does
not vanish. 
 
As expression \for{eq:dsigma} shows,  $d\s_\aa$ is a basic 3-form
with respect to the bundle $\cA\times\tenne\to\cA$. Correspondingly,
we will regard the tensor field $C$ as defined on $\cA$. Thus, the
argument used in the proof of Lemma \ref{l:upperbound} implies that

\begin{lemma}\label{l:kerC} At a point $a\in\cA$ at which
$C(a)\not=0$,
$$
   \ker C(a) \,:=\, \{u\in\renne \,:\; C_{ijk}(a)u_k=0\}
$$
is a subspace of $\renne$ of dimension $\le n-3$.
\end{lemma}

As in the standard symplectic-Hamiltonian case, the action-angle
coordinates relative to a given fibration $\pi$ need not be defined
globally and are not unique. But exactly as in that case, any two
different sets $(a,\alpha)$ and $(\tilde a,\tilde \alpha)$ of
action-angle coordinates with overlapping domains are related to each
other by transformations of the form
\begin{equation}
\label{CambioAA}
  \tilde a=Za+z \,,\qquad \tilde \alpha=Z^{-T}\alpha + \cF(a)
\end{equation}
for some unimodular matrix $Z$ with integer entries, some $z\in\renne$
and some invertible map~$\cF$~\cite{FS1}.

\begin{definition}
\label{def:CI}
Given a submersion $\pi$ as in Proposition \ref{p:CI}, a function $h$
on $M$ is called {\rm completely integrable with respect to $\pi$} if
it is in involution with all functions $\pi_1,\ldots,\pi_n$:
$$
  \{h,\pi_i\} = 0 \,,\qquad i=1,\ldots,n \,.
$$
\end{definition}

Since the $\pi_i$'s depend only on the actions, the involutivity
conditions of Definition~\ref{def:CI} and expression (\ref{eq:APB}) 
imply that {\it a function $h$ is completely integrable with respect
to $\pi$ if and only if it is a function of the actions alone}.
Moreover, under such hypothesis, by (\ref{eq:HVF}), 
\begin{equation}\label{eq:CIHVF}
  X_h(a) = - \sum_{i=1}^n \der h {a_i} (a) \partial_{\alpha_i} 
\end{equation}
so that the flow of $h$ is linear on the tori $a=\mathrm{const}$.
Furthermore, {\it every completely integrable function is strongly
Hamiltonian}: if $h=h(a)$, then $i_{X_h}d\s_\aa=0$ because $d\s_\aa$
contains no differential of the angles.

From the point of view of complete integrability the case $n=2$ is
special, and has no interest: 

\begin{proposition} \label{p:n=2}
If on a 4-dimensional almost-symplectic manifold $(M,\s)$ there is a
submersion $\pi=(\pi_1,\pi_2)$ as in Proposition~\ref{p:CI}---and
hence a completely integrable Hamiltonian system---then $\s$ is
symplectic.
\end{proposition}

Indeed, every totally antisymmetric 3-tensor on a 2-dimensional space
is identically zero; hence $C=0$ and $d\s_\aa=0$ in the domain of any
system of action-angle coordinates. 

Therefore, from now on we will assume $n\ge3$. 

\vspace{3ex}\noindent
{\it Remarks. }
(i) Restricted to the domain of an action-angle chart, completely
integrable almost-symplectic systems are dynamically
indistinguishable from the completely integrable systems of the
standard symplectic case. Even more so, the restriction
\for{eq:CIHVF} of a completely integrable almost-symplectic vector
field to a domain of action-angle coordinates is  Hamiltonian with
respect to the symplectic structure $d\alpha_i\wedge d\alpha_i$
\cite{FS1}. (It is not known if things might be different globally,
that is, if there is an almost-symplectic manifold $(M,\s)$ with a
strongly Hamiltonian vector field that is not Hamiltonian with
respect to {\it any} symplectic form on $M$).

(ii) The conclusions of Lemma \ref{l:kerC} and Proposition \ref{p:n=2}
may be reached in more geometric terms. The map $\pi:M\to \pi(M)$ is a
fibration with fiber $\toro^n$. The transition functions
\for{CambioAA} among the local systems of action-angle coordinates
show that there is a symplectic form $\s_s$ on $M$ with local
representatives $da_i\wedge d\alpha_i$. In these charts the bundle
map $\pi:M\to \pi(M)$ is $(a,\alpha)\mapsto a$ and (\ref{eq:sigma})
shows that there is a 2-form $\mu$ on $\pi(M)$ such that $\sigma =
\sigma_s + \pi^*\mu$. Hence $d\sigma = \pi^* d\mu$. If $n=2$,
$d\mu=0$.

(iii) Reference \cite{FS1} considers a more general situation of that
described in Proposition \ref{p:CI}, which extends from the
symplectic to the almost-symplectic context not only the notion of
complete integrability, but also that of `noncommutative
integrability' or `superintegrability' (in which the invariant tori
may be isotropic, not just Lagrangian).

\section{Nearly-integrable almost-symplectic Hamiltonian systems}
\subsection{Hamiltonian perturbations} 

Our goal in this paper is to study small perturbations of an
almost-symplectic completely integrable system $(M,\s,h)$ with $n\ge
3$ degrees of freedom. 

Specifically, we aim to investigate the persistence of the invariant
tori of the unperturbed system and the existence of bounds on the
variations of the actions on finite but long time scales. Therefore,
our approach may be consistently done in a neighbourhood of an
invariant torus of the unperturbed system, that is, as we will say
`semi-globally'. In particular, we may consistently restrict the
analysis to the domain $\cA\times\tenne$ of a set of action-angle
coordinates $(a,\alpha)$, with $\cA$ connected.

Thus, from now on we will restrict our study to Hamiltonian systems
of the type
\begin{equation}
\label{eq:h+epsf}
  h(a) + \epsilon f(a,\a ) \,,\qquad (a,\a)\in\cA\times\tenne
\end{equation} 
where $\e$ is a small parameter, $h$ and $f$ are two functions, and
the almost-symplectic 2-form $\s_\aa$ on $\cA\times\tenne$ is as in
\for{eq:sigma}. As is typical in perturbation theory, we will work in
the real analytic category.

It is interesting to note that, independently of the smallness of the
parameter $\e$, the dynamics of these systems is subject to
constraints that appear to come from the almost-symplectic geometry
of the manifold:

\begin{proposition}
\label{p:volume-2}
Any Hamiltonian vector field on the almost-symplectic manifold
$(\cA\times\toro^n,\s_\aa)$, with $\cA\subseteq\reali^n$ and $\s_\aa$
as in \for{eq:sigma}, preserves the volume $\sigma^n$.
\end{proposition}

\begin{proof} Assume $X_f$ is Hamiltonian. Then $L_{X} ({\s_\aa})^n =
({\s_\aa})^{n-1}\wedge L_{X} \s_\aa = ({\s_\aa})^{n-1}\wedge
i_{X}d\s$ because $i_{X}\s_\aa$ is closed. Since $\s_\aa=\sum_{i=1}^n
da_i\wedge d\alpha_i +\frac12 \sum_{i,j=1}^n A_{ij}da_i\wedge da_j$
and $i_{X_f} d\s_\aa =  \sum_{i,j,k=1}^n C_{ijk}\frac{\partial
f}{\partial \alpha_k}\ da_i\wedge da_j$, $({\s_\aa})^{n-1}\wedge
i_{X_f}d\s_\aa$ is a sum of terms each of which contains the wedge
product of at least $n+1$ differentials of the $n$ actions, and
therefore vanishes.
\end{proof}

This seems to imply that Hamiltonian vector fields on an
almost-symplectic manifold that hosts a completely integrable system
are, under certain aspects, special, not generic, among all
Hamiltonian vector fields on almost-symplectic manifolds. 

\subsection{Strongly Hamiltonian perturbations} 

Clearly, there are no obstructions whatsoever to the existence of
nearly-integrable Hamiltonian systems \for{eq:h+epsf}, because any
function $f(a,\alpha)$ gives one. 

As we now discuss, there are instead much stronger conditions on the
properties of strongly Hamiltonian nearly-integrable systems, to the
point that it is not even clear if there exist any such system which
either is not completely integrable or that does not reduce, in a
sense that will be made precise below, to a standard
nearly-integrable symplectic-Hamiltonian system.

We have not been able to investigate in full generality the structure
of all strongly Hamiltonian functions on almost-symplectic manifold
of the particular type  $(\cA\times\tenne,\s_\aa)$. We will do this
only under certain hypotheses of genericity on the strongly
Hamiltonian function $f$ and on the almost-symplectic 2-form
$\s_\aa$. 

Preliminarily to this analysis, we recall an almost-symplectic
version of the standard symplectic reduction procedure
\cite{LibMarle} studied in \cite{BlvdS}, who consider the more
general almost-Dirac case, and in \cite{vaisman}. Following
\cite{BlvdS} we will say that an  action $\P$ of a Lie group $G$ on
an almost-symplectic manifold $(M,\s)$ is {\it strongly Hamiltonian}
if its infinitesimal generators are strongly Hamiltonian vector
fields.\footnote{In the standard symplectic case, this term is
sometimes used with a different meaning (e.g. in~\cite{LibMarle}). }
Clearly, any strongly Hamiltonian action has a momentum map
$J:M\to\mathfrak g^*$, with $\mathfrak g$ the Lie algebra of $G$.
Such a momentum map is constant along the flow of any strongly
Hamiltonian system whose Hamiltonian is invariant under $\P$
\cite{BlvdS}. Denote now by $G_\mu$ the isotropy group of
$\mu\in\mathfrak g$ relative to the coadjoint action of $G$. 

\begin{proposition} 
\label{p:vaisman}
{\rm \cite{BlvdS,vaisman}} 
Consider a strongly Hamiltonian action $\P$ on an almost-symplectic
manifold $(M,\s)$ whose momentum map $J$ is equivariant with respect
to the action $\P$ on $M$ and to the coadjoint action on $\mathfrak
g^*$. Let $\mu\in\mathfrak g^*$ be a regular value of $J$ and assume
that the $\P$-action on $J^{-1}(\mu)$ is free and proper. Let $\pi:
J^{-1}(\mu) \to J^{-1}(\mu) / G_\mu$ be the canonical projection and
$i:J^{-1}(\mu) \hookrightarrow M$ the immersion. Then
\bList
\item[(i)] The smooth manifold 
$J^{-1}(\mu)/ G_\mu$ has an almost-symplectic (or symplectic)
structure $\bar\s_\mu$ such that $\pi^*\bar\s_\mu=i^*\s$. 
\item[(ii)] If $f$ is a strongly Hamiltonian function on $(M,\s)$,
then the function $\bar f$ such that $\pi^*\bar f=f$ is a strongly Hamiltonian
function on $(J^{-1}(\mu) /G_\mu, \bar \s_\mu)$.
\eList
\end{proposition}

\begin{proof} The proof of item (i) is given in \cite{vaisman}. As for
item (ii), which is not noticed in \cite{vaisman}, $0=L_X\s =
L_{\pi^* X_{\bar f}}\pi^*\bar\s = \pi^* L_{X_{\bar f}}\bar\s$. Thus
$L_{X_{\bar f}}\bar\s=0$.
\end{proof}

\subsection{The Fourier spectrum of a strongly Hamiltonian function}
\label{ss:Fourier}
 The origin of the obstruction to the existence of strongly
Hamiltonian perturbations traces back to Lemma~\ref{l:upperbound},
according to which it is always possible to choose the coordinates,
at least locally, in such a way that a strongly Hamiltonian function
is independent of at least $3$ coordinates. Lemma \ref{l:upperbound}
does not guarantees that these coordinates may be chosen to be
action-angle coordinates, and that the strongly Hamiltonian function 
is independent of (at least) three angles, but we will show that this
happens under certain conditions, and that it has further
consequences.

In order to investigate this question we will resort to Fourier series
techniques. Any function $f:\cA\times \tenne\to\reali$ can be
expanded in the Fourier series
\beq{Fourier1}
  f(a,\alpha) = \sum_{\nu\in{\interi^n}} \hat f_\nu(a) 
  E_\nu(\a) 
\eeq
where $E_\nu(\a) = e^{\sqrt{-1}\,\nu\cdot\alpha}$. We call
``spectrum'' of $f$ at a point $a\in\cA$ the set
$$
  \mathrm{Sp}(f,a) \,:=\,
  \{ \nu\in\interi^n \,:\; \hat f_\nu(a)\not=0 \} \,.
$$
The following Lemma gives a link between the spectrum of a strongly
Hamiltonian perturbation and the kernel of the 3-tensor $C$ defined
in \for{eq:C}. 

\begin{lemma}
\label{l:shf}
Consider a strongly Hamiltonian function $f$ on the
almost-symplectic manifold $(\cA\times \tenne,\s_\aa)$, with 
$\cA\subseteq\renne$ and $\s_\aa$ as in \for{eq:sigma}. Then
$$
  \mathrm{Sp}(f,a) \subseteq \interi^n \cap \ker C(a) 
  \qquad \forall a\in\cA \,.
$$
\end{lemma}

\begin{proof} Since $i_{X_f}d\s_\aa=  \sum_{i,j,k=1}^n
C_{ijk}X_f^{a_k}da_i\otimes da_j$, the condition for $f$ to be
strongly Hamiltonian is
\begin{equation}
\label{Cderfa=0}
   \sum_{k=1}^n C_{ijk}(a)\frac{\partial f}{\partial\alpha_k}(a,\alpha) =0 
   \qquad
   \forall\  i,j=1,\ldots,n\,,\ a \in\cA_C 
   \,,\ \alpha\in \toro^n
\end{equation}
that is,
\begin{equation}\label{Cderfa=0.2}
   \frac{\partial f}{\partial\alpha}(a,\alpha) \in \ker C(a) 
   \qquad \forall a \in\cA
   \,,\ \alpha\in \toro^n \,.
\end{equation}
Expanding $f$ in Fourier series, conditions \for{Cderfa=0}
become
$$
 \sum_{\nu\in{\interi^n}} \sum_{k=1}^n 
  C_{ijk}(a)\nu_k \hat f_\nu(a) E_\nu(\a) =0 
 \qquad \forall i,j=1,\ldots,n\,,\; a\in\cA
  \,,\, \alpha\in\tenne
$$
that is,
$$
 \sum_{k=1}^n C_{ijk}(a)\nu_k \hat f_\nu(a) = 0 
 \qquad \forall i,j=1,\ldots,n\,,\; a\in\cA
$$
Thus, for each $\nu\in\interi^n$ and at each point $a\in\cA_C$, if
$\hat f_\nu(a)\not=0$ then $\sum_{k=1}^n C_{ijk}(a)\nu_k = 0$ for all
$i,j$, namely $\nu\in\ker C(a)$. \end{proof}

\subsection{Constraints on strongly Hamiltonian perturbations}
\label{ss:SHPerturbations}

For systems with 3 degrees of freedom, Lemma \ref{l:shf} has the
following immediate consequence:

\begin{proposition}
\label{p:n=3}
Let $f$ be a strongly Hamiltonian function on
$(\cA\times\tenne,\s_\aa)$ with $n=3$. Assume that the basic 3-form
$d\s_\aa$  is everywhere nonzero in an open and dense subset of
$\cA$. Then 
\bList
\item[i.] $f$ is independent of the angles $\a$.
\item[ii.] $X_f$ is Hamiltonian with respect to the
symplectic structure $\sum_{i=1}^3da_i\wedge d\a_i$ on
$\cA\times\toro^3$ (and, moreover, completely integrable
with respect to $\pi:(a,\a)\mapsto a$).
\eList
\end{proposition}

\begin{proof}
i. Let $\cA_C$ be the subset of $\cA$ where $d\s_\aa$, and hence
$C$, are not zero. By Lemma \ref{l:kerC}, since $n=3$, the kernel of
$C$ is zero-dimensional at all points of $\cA_C$. Hence condition
\for{Cderfa=0.2} implies that $\der f\alpha(a,\alpha)=0$ for all
$a\in \cA_C$ and $\alpha\in \toro^n$. By continuity, $\der f\alpha=0$
in all of $\cA\times\toro^n$.

ii. This has already been noticed in Remark (i) at the end of
Section \ref{ss:SHCIS}. 
\end{proof}

The case with $n\ge 4$ is less clear and we will study it by
supplementing the hypothesis of the density of the non-zero set of
$d\s_\aa$ with conditions of genericity of the function $f$. We will
consider two such conditions.

First, we make an assumption on $f$ which is a well known condition
introduced by Poincar\'e in his study of the non-existence of first
integrals in nearly integrable Hamiltonian systems (\cite{poincare},
vol. 1, cap.~5; see also \cite{bfgg}). We say that a function
$f:\cA\times\tenne\to\reali$ is {\it Fourier-generic} in $\cA$ if for
any $\bar\nu\in\interi^n\setminus\{0\}$, either $\hat f_{\bar\nu}=0$
or, for any $a\in\cA$, there exists a $\nu\in \interi^n$ which is
`parallel' to $\bar\nu$ and is such that
$$
  f_{\nu}(a)\not=0 \,.
$$
By saying that two vectors $\nu$ and $\bar\nu$ of $\interi^n$ are
`parallel' we mean that $\nu=k\bar \nu$ for some
$k\in\razionali\setminus\{0\}$.

We note that the property of being Fourier-generic is independent of
the choice of action-angle coordinates: that is, if it is satisfied
by a function $f$, it is also satisfied by $\tilde f:=f\circ\cC^{-1}$
with $\cC:\cA\times\tenne\to\tilde\cA\times\tenne$ any change of 
action-angle coordinates, which has the form \for{CambioAA}. Indeed,
the Fourier components of the two functions $f$ and $\tilde f$ are
related by 
$$
    \hat{\tilde f}_\nu(Z a+z) = 
    e^{-\sqrt{-1}\, \nu\cdot\cF(a) } 
    \hat{f}_{Z^{-1}\nu}(a)
$$
(see \cite{fasso-hpt}) and the linear map $Z$ preserves the
`parallelism' of integer vectors.

\begin{lemma}
\label{l:n>3}
Consider a strongly Hamiltonian function $f$ on the
almost-symplectic manifold $(\cA\times \tenne,\s_\aa)$ with $n\ge4$.
Assume that $d\s_\aa$ is everywhere nonzero in an open and dense subset 
$\cA_C$ of $\cA$ and that $f$ is Fourier-generic in $\cA$.

Then there is a change of action-angle coordinates
$\cC: \cA\times\tenne \to \tilde \cA\times\tenne$,
$(a,\a)\mapsto(\tilde a,\tilde\a)$ of the type \for{CambioAA} such
that $f\circ\cC^{-1}$ depends on at most $n-3$ angles $\tilde \a$. 

\end{lemma}

\begin{proof} By Lemma \ref{l:kerC}, at each point of $\cA_C$ the
kernel of $C$ has dimension $\le n-3$. Since $f$ is strongly
Hamiltonian, then by Lemma \ref{l:shf}
$$
  \mathrm{Sp}(f,a) \subseteq \cL_a  \qquad \forall a\in\cA_C \,,
$$
with
$$
  \cL_a \;=:\; \interi^n \cap \ker C(a) \,.
$$
Since at the points $a$ of $\cA_C$, $\ker C(a)$ is a subspace of
$\renne$ of dimension $\le n-3$, at each of these points the set
$\cL_a$ is a sublattice of $\interi^n$ of rank $r\le n-3$. 

We now observe that, under the hypotheses of item ii., if we define
$$
  \cL \;:=\; 
  \bigcap_{a\in\cA_C}\cL_a \ug
  \Big(\bigcap_{a\in\cA_C}\ker C(a) \Big) \bigcap \interi^n 
$$
then we have
$$
  \Sp{f,a} \subseteq   \cL \qquad \forall a \in \cA_C\,.
$$ 
Indeed, assume that $\bar\nu\in \Sp{f,\bar a}$ for some $\bar
a\in\cA_C$, so that $\bar\nu\in\cL_{\bar a}$. Since $f$ is
Fourier-generic in $\cA$, and hence in $\cA_C$, for any $a\in\cA_C$
there exists $\nu\in\interi^n$ parallel to $\bar\nu$ such that
$\nu\in\Sp{f,a}$ and hence $\nu\in\cL_a$. But $\cL_a$, being the
intersection of $\interi^n$ with a subspace of $\renne$, contains all
integer vectors parallel to $\nu$. Thus $\bar\nu\in\cL_a$. This
proves that $\bar\nu\in\cap_{a\in\cA_C}\cL_a$.

Since each $\ker C(a)$ with $a\in \cA_C$ is a subspace of $\renne$ of
dimension $\le n-3$, the intersection $\cap_{a\in\cA_C}\ker C(a)$ is
also a  subspace of $\renne$ of dimension $\le n-3$ and $\cL$ is a
sublattice of $\interi^n$ of rank $r\le n-3$. 

Now, a sublattice of $\interi^n$ of rank $r$ is the set of all the
linear combinations with integer coefficients of $r$ vectors
$u_1,\ldots,u_r\in \interi^n$, called a basis. Consider a basis 
$\{u_1,\ldots,u_r\}$ of $\cL$ and complete it to a basis
$\{u_1,\ldots,u_r,u_{r+1},\ldots,u_n\}$ of $\interi^n$. That this is
possible is guaranteed by the Elementary Divisor Theorem (see e.g.
\cite{lang}, Theorem 7.8) thanks to the fact that $\cL$ is not just a
generic sublattice of $\interi^n$, but it is the intersection of
$\interi^n$ with a subspace of $\renne$.

Specifically, the Elementary Divisors Theorem states that for any 
finitely generated submodule $\not=\{0\}$ (e.g., a lattice) of a free
abelian module over a principal ideal domain (e.g., $\interi^n$)
there exists a basis $\{u_1,\ldots,u_n\}$ of the latter, an integer
$1\le r\le n$ and integers $d_1,\ldots, d_r$ such that
$\{d_1u_1,\ldots,d_ru_r\}$ is a basis of the former. In our case, all
$d_i=1$ because the lattice $\cL$ is the intersection of
$\interi^n$ with a subspace of $\renne$.

Since $\{u_1,\ldots,u_r\}$ is a basis of $\interi^n$ there exists a
unimodular integer matrix $Z$ such that $Zu_i=e_i$, the $i$-th unit
vector, for all $i=1,\ldots,n$. The change of coordinates
$$
  \cC_C: \cA_C\times\tenne \to \tilde\cA_C\times \tenne \,,
  \qquad (a,\a)\mapsto (\tilde a,\tilde \a) = (Z^Ta, Z^{-1}\a) 
$$
produces a new set of action-angle coordinates with the property that
the spectrum of the representative $\tilde f$ of $f$ is contained in the
lattice generated by $e_1,\ldots,e_r$. Hence, $\tilde f$ depends only on
the first $r\le n-3$ angles. The change of action-angle coordinates
$\cC_C$ extend by linearity to a change of action-angle coordinates
$\cC$ which is defined in all of $\cA\times\tenne$ and, by
continuity, conjugates $f$ to a function that depends only on the
first $r$ angles.
\end{proof}

Lemma \ref{l:n>3} has the following consequence:

\begin{proposition}
\label{p:n>3}
Let $f$ be a strongly Hamiltonian function on
$(\cA\times\tenne,\s_\aa)$ with $n\ge4$. Under the hypotheses of
Lemma \ref{l:n>3} the system reduces, under a torus action, to a
family of (possibly nonintegrable) symplectic-Hamiltonian systems
with at most $n-3$ degrees of freedom.
\end{proposition}

\begin{proof}
Let $f(a,\a)$ be a strongly Hamiltonian function. By Lemma
\ref{l:n>3}, there is a choice of  action-angle coordinates such that
$f$ is independent of the last $r\ge 3$ angles. We denote by $(a,\a)
= (I,J,\varphi,\psi)\in \reali^{n-r}\times\reali^{r} \times
\toro^{n-r}\times\toro^{r}$ these coordinates, with $f$ independent
of the $r$ angles $\psi$. The system is thus invariant under the
$\toro^r$-action given by translations of the angles $\psi$. This is
a strongly Hamiltonian action, with equivariant momentum map given by
the actions $J:M\to \reali^r$. Fix a value of $J\in\reali^r$. Then
the reduced phase space is $\bar\cA_J\times \toro^{n-r} \ni (I,\p)$ 
with $\bar \cA_J \subseteq \reali^{n-r}$ and, if we write $A = 
\left(\begin{matrix} \bar A & B \cr -B^T & D \cr \end{matrix}\right)$
with the block $\bar A$ of dimension $(n-r)\times(n-r)$ etc., the 
reduced almost-symplectic 2-form is
$$
   \bar{\s_\aa}_J  = 
   \sum_{i=1}^{n-r} dI_i\wedge d\varphi_i 
   +
   \frac12 \sum_{i,j=1}^{n-r} (\bar A_J)_{ij} dI_i\wedge dI_j 
$$
with $\bar A_J(I):=\bar A(I,J)$. The reduced Hamiltonian is  $\bar
f_J(I,\p):=f(I,J,\p)$. 

There are now three possibilities. If $\bar{\s_\aa}_J$ is symplectic
(what happens, in particular, if $n-r=0,1,2$), then the reduced system
is symplectic-Hamiltonian. If $n-r=3$ and $d\bar{\s_\aa}_J\not=0$, then
in view of Proposition \ref{p:n=3} the reduced system is
symplectic-Hamiltonian with respect to a modified symplectic form.

The last possibility is that $n-r\ge 4$ and $d\bar{\s_\aa}_J\not=0$.
Note that the function $f_J$ inherits from $f$ the property of Fourier
genericity and $d\bar{\s_\aa}_J$ inherits from $d\s_\aa$ the property
of vanishing in a subset of the reduced action space $\bar A_J$ whose
complement is open and dense. Therefore, the reduced system
$(\bar\cA_J\times\toro^{n-r},\bar{\s_\aa}_J,\bar f_J)$ satisfies the
hypotheses of the present Proposition and we may apply to it the
reduction procedure just described. This leads to a family of reduced
systems with at most $n-6$ degrees of freedom, each of which is
either symplectic, or has $3$ degrees of freedom, or has more than 3
degrees of freedom and is almost-symplectic. The reduction procedure
can be applied to the latter, etc. The iteration stops when all
reduced systems obtained either are symplectic-Hamiltonian (in
particular, if they have $0$,$1$ or $2$ degrees of freedom) or have
3 degrees of freedom.
\end{proof}

The same conclusions about the structure of strongly Hamiltonian
functions can be obtained under different hypotheses on these
functions. Let $\Sp f=\{\nu\in\interi^n \,:\;
f_\nu\not=0\} = \cup_{a\in\cA}\Sp{f,a}$. For each $\nu\in\Sp f$, the
set
$$
  \cF_\nu := \{a\in \cA_C : f_\nu(a)\not=0 \}
$$
is open and nonempty, and $\nu\in\ker C(a)$ for all $a\in\cF_\nu$. 
The set
$$
   \cF := \bigcap_{\nu\in\Sp f }\cF_\nu 
$$
need not be open and nonemtpy. However, if it is nonempty, then any
$\nu\in\Sp f$ satisfies $\nu\in\ker C(a)$ for all $a\in\cF$ and hence 
$\nu\in \big(\cap_{a\in\cF}\ker C(a)\big)\cap\interi^n$.
Thus
$$
  \Sp{f} \subseteq 
  \Big(\bigcap_{a\in\cF}\ker C(a)\Big)\bigcap\interi^n \,. 
$$

\begin{proposition}
\label{p:n>3bis}
The conclusion of Proposition \ref{p:n>3} remains true if it is
assumed that $\cA_C$ is dense in $\cA$ and, instead of the Fourier
genericity of $f$, that for each $\nu\in\Sp f$ the set $\cF_\nu$ is
dense in~$\cA$. 
\end{proposition}

\begin{proof}
$\cF$ is a countable intersection of open dense subsets of~$\cA$.
Since $\cA$ is an open subset of $\renne$, a straightforward
application of the Baire category theorem guarantees that $\cF$ is an
open dense subset of $\cA$. The set $\Big(\bigcap_{a\in \cF}\ker
C(a)\Big) \cap \interi^n$ is a sublattice of $\interi^n$ of rank
$r\le n-3$ and is the intersection of $\interi^n$ with a subspace of
$\renne$. We may thus proceed as in the proof of Lemma \ref{l:n>3}
and conclude that there is a system of action-angle coordinates
$(I,J,\p,\psi)$ in which $f$ depends only on the $r\le n-3$ angles
$\psi$. The statements as in Proposition \ref{p:n>3} follow from here.
\end{proof}

Combining the arguments used in the proofs of Lemma \ref{l:n>3} and
Proposition \ref{p:n>3bis}, one easily sees that the conclusions of
Proposition \ref{p:n>3} remain valid if the Fourier-genericity of the
function $f$ is weakened, by assuming that, for each $\bar\nu\in\Sp f$,
for each $a$ in a dense subset of $\cA$ there is a $\nu$ parallel to
$\bar\nu$ such that $f_\nu(a)\not=0$.

In view of Propositions \ref{p:n=3} and \ref{p:n>3}, and at least
under the stated hypotheses, strongly Hamiltonian perturbations $h+\e
f$ of completely integrable almost-symplectic systems reduce to
families of symplectic-Hamiltonian nearly integrable systems 
and can therefore be studied via the standard results and techniques
of Hamiltonian perturbation theory, applied to each reduced
system. We will therefore restrict our study of perturbation theory
for nearly-integrable almost-symplectic systems to the case of
perturbations that are not strongly Hamiltonian. 

We remark that, even though it reduces to a family of
symplectic-Hamiltonian systems, the non-symplectic-Hamiltonian
character of the system is encoded in the evolution of the angles
that have been quotiented out in the reduction process. The
corresponding `reconstruction' equation  is given by the equations of
motion of the angles $\psi$, which using the notation of the proof of
Proposition \ref{p:n>3} is
\beq{reconstruction}
  \dot \psi = \der f J(I,J,\varphi) - B(I,J) \der f\varphi
  (I,J,\varphi) \,,
\eeq
and by the analogous equations at the other stages of the reduction
procedure. 

\vspace{3ex}\noindent
{\it Examples. } Examples of the situation described in this Section
are easily constructed. For instance, the matrix
$
A=\left(
  \begin{matrix} 
     0 &a_4 &0 &0 \cr
     -a_4 &0 &0 &0 \cr
     0 &0 &0 &0 \cr
     0 &0 &0 &0
  \end{matrix}
  \right)
$
leads to an almost-symplectic structure on $\reali^ 4\times \toro^4$
of the form (\ref{eq:sigma}) that is not symplectic. It is immediate
to check that the quantities $C_{ijk}\frac{\partial f}{\alpha_k}$
either are $0$ or, up to the sign, equal one of the three derivatives
$\frac{\partial f}{\alpha_1}$, $\frac{\partial f}{\alpha_1}$,
$\frac{\partial f}{\alpha_3}$. Thus a function is strongly
Hamiltonian function if and only if it is independent of the three
angles $\alpha_1, \alpha_2, \alpha_3$. An example with $n=5$ has
$
A=\left(
  \begin{matrix} 
     0 &a_1a_3 &0 &0 &0 \cr
     -a_1a_3 &0 &0 &0 &0 \cr
     0 &0 &0 &0 &0 \cr
     0 &0 &0 &0 &0 \cr
     0 &0 &0 &0 &0 
  \end{matrix}
  \right)
$;
here too, the strongly Hamiltonian functions are the functions
independent of $\alpha_1,\alpha_2,\alpha_3$; among them,
$\frac{a_4^2}2+a_5-(1+\cos\alpha_5)\cos\alpha_4$ describes a
periodically perturbed pendulum, which is nonintegrable.

\section{Perturbation theory }

\subsection{A first look}
\label{ss:FirstLook}
We begin by investigating the possibility of a perturbation theory
for nearly integrable almost-sympletic Hamiltonian systems, so as to
determine the analogies and the differences from the standard
symplectic-Hamiltonian case. Our treatment, at this initial stage,
will be rather formal. More precise considerations will be made in
subsection~\ref{ss:AS-Nekh}. 

\newcommand\K{X_k}
\newcommand\F{X_f}
\newcommand\X{Z}

We start from the system 
\beq{K+F}
  \K + \e \F \,,
\eeq
with $k=k(a)$ and $f=f(a,\a)$, on the phase space
$\cA\times\tenne\ni(a,\a)$ with $\cA\subseteq\renne$ and
almost-symplectic form $\s_\aa$ as in \for{eq:sigma}. We assume $n\ge
3$. Moreover, we assume $k$ and $f$ to be real analytic,
and $\e$ (suitably) small. 

The equations of motion of system \for{K+F} are
$$
  \dot a  = -\e \der f\a(a,\a) \,,\qquad 
  \dot \a =  \der ka(a) + \e \der fa(a,\a)  + \e A\der f\a(a,\a)  \,.  
$$
The first of these equations gives the apriori estimate
$|a_t-a_0|=\cO(\e t)$ on the variation of the actions over a time
$t$, and hence
$$
  |a_t-a_0| \le \const \, \e^{c_1} 
  \qquad \mathrm{for} \quad 
  |t| \le \const\, \e^{-c_2}
$$
with any pair of positive constants $c_1$ and $c_2$ such that
$c_1+c_2=1$. The goal of perturbation theory is to go beyond this
apriori estimate. 
 
Since the $a$-equation for the vector field \for{K+F} is the same as
that of the symplectic-Hamiltonian case, it can be expected that it
might be possible to build a perturbation theory which is to some
extent similar to the symplectic-Hamiltonian one. The basic step is
to look for the existence of a family of diffeomorphisms $\P_\e$
which depends smoothly on~$\e$ in an interval which contains
zero, equals the identity for $\e\to0$ and is such that
$$
  \P_\e^*(\K+\e \F)  =  \K + \e G + \cO(\e^2)
$$
with a vector field $G$ which is ``as integrable as possible'' or
that, at least,  moves the new actions $a\circ\P^{-1}_\e$ as little
as possible. If this is the case, then the vector field $\K + \e G +
\cO(\e^2)$  will be generically called a normal form.\footnote{The
fact that the remainder is order $\e^2$ is formal: due to the
presence of resonances, the remainder might in fact be $\cO(\e^p)$
with some $1<p<2$, see below. To simplify the exposition, however, in
this Section we adopt this formal point of view.} This procedure
should then be  iterated as many times as possible.  We begin by
looking at the first step.

Preliminarily, we recall that in the symplectic case the
diffeomorphisms $\P_\e$ are constructed so as to be symplectic, and
the normal form is accordingly built for the Hamilton function,
rather than for the Hamiltonian vector field. However, in the
almost-symplectic context there is no analog of a symplectic
transformation, which conjugates Hamiltonian vector fields to
Hamiltonian vector fields while conjugating as well the respective
Hamiltonian functions. Hence, we are forced to work with the
Hamiltonian vector fields.

A standard way of constructing the family of diffeomorphisms $\P_\e$
is through the maps $\P^\X_\e$ at time $\e$ of the flow $\P^\X$ of a
vector field $\X$. This is the so called Lie method, that we will
apply to vector fields, see \cite{fasso-VF} for details. Recalling the
basic identity $\frac d{dt} (\P^\X_t)^* Y = (\P^\X_t)^* (L_\X Y)$
between the pull back of a vector field $Y$ under a flow and the Lie
derivative (here we write $L_\X Y$ for $[\X,Y]$), one immediately
sees that
$$
  \P_\e^* Y \ug Y + \e R_\e^1(Y) \ug Y + \e L_\X Y + \e^2 R_\e^2(Y) 
$$
where, if both $\X$ and $Y$ are real analytic,
$$
 R_\e^1(Y) = \sum_{s=1}^\infty \frac{\e^{s-1}} {s!}L_\X^s Y
 \,,\qquad
 R_\e^2(Y) = \sum_{s=2}^\infty \frac{\e^{s-2}} {s!}L_\X^s Y
$$
with $L^1_\X Y=L_\X Y$ and $L^{s+1}_\X Y=L_\X^s(L_\X Y)$ for $s\ge1$.

Applying the Lie method to \for{K+F}, and observing that
$R_\e^1(X_{\e f}) =  \e R_\e^1(\F)$, gives
$$
  \P_\e^* (\K+\e \F) = 
  \K + \e [\X,\K] + \e \F + \e^2 R_\e^2(\K) + \e^2 R_\e^1(\F) 
$$
and therefore, given that the last two terms are $\cO(\e^2)$,
the vector field $\X$ should be selected so that
\beq{VectHomEq}
  [\X,\K] + \F = G 
\eeq
with some $G$ with the desired properties. Equation \for{VectHomEq}
is the so-called {\it (`vector') homological equation} of
perturbation theory. There are very well known obstructions to the
existence of solutions to this equation, due to the presence of
resonances, and it is well known that, in order to obtain a solution,
the equation has to be modified. 

Before seeing this, we point out that if we look for solutions $\X,G$
of equation \for{VectHomEq} which are Hamiltonian, then we are
essentially in the standard symplectic-Hamiltonian case. To see this,
we first note that, for Hamiltonian vector fields, equation
\for{VectHomEq} reduces to the standard homological equation for the
Hamiltonian functions of the symplectic case:

\renewcommand{\chi}{z}
\begin{lemma} \label{l:SolEqOm} 
If there exist functions $\chi$ and $g$ which satisfy the (`scalar')
homological equation
\beq{ScalarHomEq}
  \{k,\chi\}_\aa + f = g
\eeq
then $\X=X_\chi$ satisfies the (`vector homological')
equation \for{VectHomEq} with $G=X_g$.
\end{lemma}

\begin{proof} Since $\K$ is strongly Hamiltonian and $\X=X_\chi$ is
Hamiltonian, by Lemma \ref{l:homomorphism}
the vector field $[\X,\K]$ is Hamiltonian and equals
$X_{\{k,\chi\}_\aa}$. Therefore 
$[\X,\K]-\F-G=X_{\{k,\chi\}_\aa-f-g}=0$. 
\end{proof}

Let now $\o=\der ka$ be the frequency map of the unperturbed system,
so that $\K=\sum_{j=1}^n\o_j\partial_{\a_j}$. Since the function $k$
depends only on the actions, the $A$-dependent terms in the almost
Poisson brackets $\{k,\chi\}_\aa$ are absent, see \for{eq:APB}.
Therefore, the scalar homological equation \for{ScalarHomEq} reduces
exactly to the standard homological equation of the symplectic case,
namely
$$
  \o \cdot \der\chi\a = g -f \,.
$$
Furthermore, given that the relation between the action-components of
a Hamiltonian vector field and its Hamiltonian function is the same
as in the symplectic case, one realizes that in the almost-symplectic
case the normal form term $g$ can be chosen exactly as in the
symplectic case: namely, as (partial) average of the perturbation
$f$. We will be more precise on this in the next Subsection.

In other words, in the almost-symplectic case that we consider, at
the level of Hamiltonian functions things go exactly as in the
symplectic case. Nevertheless, even in the first normalization step
that we are considering here, there are differences at the level of
the normal form vector fields. A (minor) difference from the standard
symplectic case is that the $\a$-components of $\X$ and $G$ contain
extra $A$-dependent terms. More important, the vector field $\X$ is
Hamiltonian, but need not be strongly Hamiltonian. This has the
consequence that the remainder $\e^2 R_\e^1(\F)+\e^2R_\e^2(\K)$ need
not be Hamiltonian, and the procedure just outlined cannot be
iterated. 

The conclusion of this elementary analysis is that it can be expected
that all results on the variations of the actions that, in the
symplectic case, follow from a single normalization step will retain
their validity in the almost-symplectic case. As we will see, this
includes a `first-order' formulation of Nekhoroshev theorem. However,
all results obtained through iteration of the normal form procedure,
in particular the KAM theorem, will not extend to the
almost-symplectic case, unless the perturbation has special
properties (e.g., it is strongly Hamiltonian, see Section
\ref{ss:4.3}).

\subsection{An almost-symplectic Nekhoroshev-like theorem}
\label{ss:AS-Nekh}

In order to make more definite statements, we need to take into
considerations the role of resonances. This requires the
consideration of Fourier series of Hamiltonian vector fields and of
some properties of their (partial) averages.

From now on, we will write the Fourier series \for{Fourier1} of a 
function on $\cA\times\renne$ as
$$
   y=\sum_{\nu\in\interi^n}y_\nu \,,
$$
where the functions $y_\nu:\cA\to\reali$, that we call the {\it
harmonics} of $y$, are given by $y_\nu(a,\a) = \hat y_\nu(a)
E_\nu(\a)$. Similarly, if $Y$ is a vector field, we will write
$$
   Y=\sum_{\nu\in\interi^n}Y_\nu 
$$
where, for each $\nu$, the {harmonic} $Y_\nu$ is defined as  the
vector field whose components are the $\nu$-th harmonics of the
components of $Y$. Note that if $Y$ is Hamiltonian, with Hamiltonian
function $y$, then, for each $\nu\in\interi^n$, $Y_\nu$ is a
Hamiltonian vector field, with Hamiltonian function $y_\nu$.
Furthermore, for any subset $\L$ of $\interi^n$ we define projectors
$\Pi_\L$ on the spaces of functions and vector fields as
$$
  \Pi_\L y := \sum_{\nu\in\L}y_\nu 
  \,,\qquad
  \Pi_\L Y := \sum_{\nu\in\L}Y_\nu \,.
$$
Clearly, if $Y$ is Hamiltonian with Hamiltonian $y$, then $\Pi_\L Y$
is Hamiltonian with Hamiltonian $\Pi_\L y$.

A point $a\in\cA$ is said to be {resonant} with a vector $\nu \in
\interi^n$ if $\o(a)\cdot \nu=0$. In that case, $\nu$ is called a
{resonance} of $a$ and $|\nu|:=\sum_{i=1}^n|\nu_i|$ its {order}. The
resonances of a point $a$ form a sublattice $\L_a$ of $\interi^n$.
Conversely, given a subset (not necessarily a sublattice) $\L\subseteq
\interi^n$, the $\L$-resonant set is
$$
  \cA_\L:=\{a\in\cA: \o(a)\cdot\nu=0 \mathrm{\ for\ all\ } \nu\in\L\} \,.
$$ 

By expanding all functions in Fourier series, the scalar homological
equation \for{ScalarHomEq} becomes
$\sqrt{-1}\, \o\cdot \nu = g_\nu-f_\nu$ for all $\nu\in\interi^n$.
Hence, at each point $a$, if $\L_a$ denotes as above the set of
resonances of $a$, equation \for{ScalarHomEq} has the solution
$$
  \chi(a,\a) = 
  \sum_{\nu\notin\L_a} \frac{f_\nu(a,\a)} {\sqrt{-1}\,\o(a)\cdot\nu}
  \,,\qquad
  g(a,\a) = \Pi_{\L_a} f (a,\a) \,.
$$
(If $\L\not=\{0\}$ then this solution is not unique, because there is
arbitrariness in the choice of $\chi_\nu$ for $\nu\in\L_a$ and of
$g_\nu$ for $\nu\notin\L_a$; however, this solution is the one which
is usually considered in the symplectic case and there is no reason
here to change it). Due to the resonances, the solution above has
obvious and well known smoothness problems. 

Specifically, if, as we will assume, the Hamiltonian $k$ is such that
the frequency map $\o:\cA\to\renne$ is a local diffeomorphism, which
happens if $k$ satisfies Kolmogorov's nondegeneracy condition
$\det\dder k a a (a) \not=0$ for all $a\in\cA$, then the set of
resonant points is dense in~$\cA$. The way out depends to a certain
extent on the result one is looking for, but for both KAM and
Nekhoroshev theorems it is based on the approximation of the
perturbation by a finite order Fourier truncation
\beq{FourierTruncation}
  f^{\le N}(a,\a) \;:=\; \sum_{\nu\in\interi^n ,\, |\nu|\le N}
  f_\nu(a) \,,
\eeq
so as to have to deal with only a finite number of resonances and
avoid the density of resonances, and on the construction of
(resonant) normal forms in neighbourhoods of the corresponding
resonant sets. The parameter $N$ is called a cutoff and
$\F^{>N}:=\F-\F^{\le N}$ the ultraviolet part of $\F$.  For real analytic
vector fields, $\F^{>N}$ decays with $N$ as $\exp{(-\const \times
N)}$. Thus, suitably choosing $N$ as a function of $\e$ makes
$\F^{>N}$ of order $\e^2$ (or smaller, if needed).

We thus fix a cutoff $N$, a set $\L\subseteq \interi^n_N :=
\{\nu\in\interi^n \,:\; |\nu|\le N\}$ and a subset $\cB_\L$ of $\cA$
whose points possibly resonate with the vectors of $\L$ but do not
resonate with any other vector $\nu\in\interi^n_N\setminus \{\L\}$. 
Since 
\beq{CV-NF}
  \P_\e^* (\K+\e \F) = 
  \K + \e [\X,\K] + \e \F^{\le N}  + \F^{>N} + \e^2 R_\e^2(\K) + \e^2
  R_\e^1(\F)
\eeq
if we take $\X=X_\chi$ in $\cB_\L\times\tenne$ with
\beq{chi-g}
  \chi =  \sum_{\nu\notin\L} \frac{f_\nu} {\sqrt{-1}\, \o\cdot\nu}
  \,,\qquad g = \Pi_{\L} f 
\eeq
we obtain the $\L$-resonant normal form 
\beq{NF}
  \P_\e^* (\K+\e \F) = 
  \K + \e \Pi_\L \F^{\le N} + \cO(\e^2)
\eeq
which is defined in $\cB_\L\times \tenne$. Note that the function
$\chi$ is now a sum of finitely many terms, and all denominators are
nonzero, so there are no smoothness issues.  The usefulness of these
approximate normal forms in the standard symplectic case is due to
the properties of the $a$-components of the averages $\Pi_\L \F$.
These properties are valid in the almost-symplectic case as well:

\begin{lemma}\label{l:averages}
Let $\L$ be a subset of $\interi^n$. Then, if $Y$ is Hamiltonian, the
$a$-component of  $\Pi_\L Y$ is parallel to $\L$.
\end{lemma}

\begin{proof} If $y$ is a Hamiltonian of $Y$ then $\Pi_\L y$ is a 
Hamiltonian of $\Pi_\L Y$ and
$Y_{a}=-\der{}\a \sum_{\nu\in\L} y_\nu 
= \sum_{\nu\in\L} \sqrt{-1}\, \nu y_\nu$.
\end{proof}

{\it Remark: } The statement in Lemma \ref{l:averages} that a
vector field on $\cA$ is parallel to a set $\L\subseteq\renne$ is
meaningful because, just in the standard symplectic case,
the action space $\cA$ has an affine structure \cite{FS1}. 

\vspace{3ex}
Lemma \ref{l:averages} has the consequence that, near a resonance set
$\cA_\L$ where the dynamics is described by the normal form \for{NF},
for times short with respect to $\e^{-2}$ the motion of the
(transformed) actions takes place approximately in an affine subspace
parallel to $\L$ of the action space.  This is the `fast drift'
subspace of Nekhoroshev theory. In the symplectic Hamiltonian case, 
Nekhoroshev theory provides  a mechanisms of confinement of such fast
drift, which requires certain properties of the unperturbed
Hamiltonian. The most general among these properties are the so
called steepness properties (see \cite{GuzzoChierchiaBenettin} for a
recent, refined proof of Nekhoroshev theorem under such general
conditions), but a simple case is provided by convexity, namely
$$ 
 \Big|u\cdot \dder k a a (a)u \Big| 
 \ge \const\, \|u\|^2 \qquad \forall \, u\in\renne \,,\; a\in\cA \,.
$$
Under this hypothesis (or more generally, under the so-called
hypothesis of quasi-convexity, see \cite{poeschel}), in the symplectic
case the confinement of the actions' movement along the fast drift
hyperplane is provided by the conservation of the Hamiltonian.

In our case, the normal form vector field $(\P^\X_\e)^*(\K+\e \F) =
\K+\e\Pi_\L \F+\cO(\e^2)$ is not Hamiltonian. However, the original
system $\K+\e \F$ is Hamiltonian. As consequence, its Hamilton
function $k+\e f$ is a first integral of $\K+\e \F$. In turn,
the function  $(\P^\X_\e)^* (k+\e f)$ is a first integral of  the
vector field $(\P^\X_\e)^*(\K+\e \F)$. This first integral provides
the necessary confinement. 

The only, real difference from the symplectic case is that the
procedure cannot be iterated, and `exponentially long' time scales are
not reached. However, these considerations should make clear that the
following result can be reached instead:

\begin{proposition}\label{Nek-as}
Consider the system of Hamiltonian $k(a)+\e f(a,\a)$ on $\cA\times
\tenne$, equipped with the almost-symplectic structure $\s_\aa$.
Assume that $k$ and $f$ are real analytic and that $k$ is convex.
Then, there exist positive constants $A$, $T$, $c_1$ and $c_2$
independent of $\e$ and with
$$
   1< c_1 + c_2 <2
$$
such that, for $\e$ sufficiently small, all motions $t\mapsto (a_t,\a_t)$
satisfy
\beq{NostreStime}
  \|a_t-a_0\| \le A \e^{c_1}   \quad\mathrm{for}\quad
  |t| \le T \e^{-c_2} \,.
\eeq
\end{proposition}

We do not give here a proof of this result because, as we have
already mentioned, the proof can be obtained with very minor
modifications of, for instance, the proof given in \cite{poeschel} for
the symplectic case. Specifically, besides some small differences in
the construction and estimate of the normal form vector field (which
is not the Hamiltonian vector field of the normal form Hamiltonian
and thus needs to be treated on its own), the main difference is that
in our case it is not necessary to iterate the construction of the
normal form; this simplification leads to different estimates on the
confinement of motions, which are however easily worked out. 

Following the argument in \cite{poeschel} one obtains, as possible
values of the two constants $c_1$ and $c_2$, for instance, $c_1=
\frac1{8n}$ and $c_2=\frac32(1-\frac1{4n})$. These values improve,
even only slightly, on the apriori estimate $c_1+c_2=1$. It is
possible that better values of these constants, particularly of
$c_2$, might be found by carefully complementing the treatment with
some specificities of the problem at hand. We also note that, since
in the standard Nekhoroshev theory real analyticity is needed only to
obtain an exponentially long time scale (see e.g. \cite{Bou}), the
result of Proposition \ref{Nek-as} remain valid (possibly with worse
values of the constants $c_1$ and $c_2$) for smooth Hamiltonians.
However, we leave these analyses for a possible future work because
the technical arguments involved are rather extraneous to the purpose
and the spirit of the present work.

\vspace{3ex}
{\it Remark: } The values of the constants $c_1$ and $c_2$ 
reported above can be obtained under the additional hypothesis that $\o$ is
uniformly bounded away from $0$ in~$\cA$; if not, slightly worse
values can be found; see \cite{poeschel} for the treatment of this
technical fact in the symplectic case.

\subsection{On the case of strong Hamiltonian  perturbations}
\label{ss:4.3}

If the perturbation $f$ is strongly Hamiltonian then, at least under
the hypotheses considered in Section \ref{ss:SHPerturbations}, it is
possible to study the reduced symplectic-Hamiltonian systems via the
standard techniques of Hamiltonian perturbation theory. Thus KAM and
Nekhoroshev theorem are valid for the reduced systems and can be
lifted to the unreduced system by means of the reconstruction
equation, see equation \for{reconstruction}. Alternatively, however,
one may apply the perturbation technique described in the previous
Sections \ref{ss:FirstLook} and \ref{ss:AS-Nekh} to the unreduced
system. At variance from the case of a perturbation that is only
Hamiltonian, if the perturbation is strongly Hamiltonian, then the
construction of the normal forms can be iterated, and the standard
KAM and Nekhoroshev theorems of the symplectic case may be recovered. 

This is due to the following fact:

\begin{proposition} 
If $f$ is strongly Hamiltonian, then the normal form vector field
\for{CV-NF}, with $\chi$ and $g$ as in \for{chi-g}, is strongly
Hamiltonian.
\end{proposition}

\begin{proof} First of all, we note that a function $y$ is strongly
Hamiltonian if and only if its harmonics $y_\nu$ are strongly
Hamiltonian. In fact, $X_y =\sum_\nu (X_y)_\nu = \sum_\nu X_{y_\nu}$
and the vanishing of
$i_{\sum_\nu X_{y_\nu}}d\s= \sum_\nu i_{X_{y_\nu}}d\s$ is equivalent to
the vanishing of each $i_{X_{y_\nu}}d\s$. 

If $\chi$ is as in \for{chi-g} then, for each $\nu$,
$\chi_\nu=\frac{ f_\nu}{\sqrt{-1}\,\o\cdot \nu}$ and 
$$
 \sum_{k=1}^n C_{ijk}\der{\chi_\nu}{\a_k} 
 \ug
 \frac{ 1 }{\sqrt{-1}\,\o\cdot \nu} \, \sum_{k=1}^n C_{ijk}\der{f_\nu}{\a_k} 
$$
which vanishes because $f_\nu$ is strongly Hamiltonian. This proves
that each $\chi_\nu$ is strongly Hamiltonian, see \for{Cderfa=0}, and
hence $\chi$ is strongly Hamiltonian.

Next, we note that {\it if $z$ and $y$ are two strongly Hamiltonian
functions, then, for any $t$, the function $(\P^{X_z}_t)_*y$ is
strongly Hamiltonian and its Hamiltonian vector field is
$(\P^{X_z}_t)_*X_y$}. The proof of this fact is immediate because,
restricted to its strongly Hamiltonian vector fields, an
almost-symplectic structure behaves as a symplectic one. 

It follows that $R_\e^1(\F)= \P_\e^* (\F)-\F$ and, taking into account
Lemma \ref{l:homomorphism} as well,  $R_\e^2(\K)= \P_\e^* (\K)-\K- \e
[\X,\K]$ are strongly Hamiltonian. The proof is concluded by noting
that, on account of what has been noticed above, $\F^{\le N}$ and
$\Pi_\L\F^{\le N}$ are strongly Hamiltonian, too. \end{proof}

\vspace{3ex}\noindent
{\small
{\bf Acknowlegments} We are grateful to Giovanna Carnovale and
Andrea Giacobbe for pointing out the Elementary Divisor Theorem to us
and to Umberto Marconi for some technical remarks on the Baire
Category Theorem.

}


\begin{thebibliography}{99}

\bibitem{AKN}
V.I. Arnold, V.V. Kozlov and A. Neishtadt,
{\em Mathematical Aspects of Classical and Celestial Mechanics.
Dynamical Systems, III. Third Edition. Encyclopaedia Math. Sci. {\bf
3}.} (Springer, Berlin, 2006).

\bibitem{bates-sn}
L. Bates and J. \'Sniatycki, {\em Nonholonomic reduction.}
Rep. Math. Phys. {\bf32} (1993), 99-115.

\bibitem{bfgg}
G. Benettin, G. Ferrari, L. Galgani and A. Giorgilli, 
{\em 
An extension of the Poincar\'e-Fermi theorem on the nonexistence of
invariant manifolds in nearly integrable Hamiltonian systems.}
Nuovo Cimento B (11) {\bf 72} (1982), 137-148. 

\bibitem{bengall}
G. Benettin and G. Gallavotti,
{\em Stability of motions near resonances in quasi--integrable
Hamiltonian systems.}
J. Statist. Phys. {\bf 44} (1986), 293-338.

\bibitem{BlvdS}
G. Blankenstein and A.J. van der Schaft,
{\em Symmetry and reduction in implicit generalized hamiltonian
systems.}
Rep. Math. Phys. {\bf 47} (2001), 57-100.

\bibitem{Bou}
A. Bounemora,
{\em Nekhoroshev estimates for finitely differentiable quasi-convex
Hamiltonians.} J. Diff. Eq. {\bf 249} (2010), 2905-2920.

\bibitem{fasso-VF}
F. Fass\`o, {\em Lie series method for vector fields and
Hamiltonian perturbation theory.} Z. Angew. Math. Phys. {\bf 41} (1990),
843-864.

\bibitem{fasso-hpt}
F. Fass\`o, {\em Hamiltonian perturbation theory on a manifold.}
Celestial Mech. Dynam. Astronom.  {\bf 62}, 43-69 (1995)

\bibitem{FS1}
F. Fass\`o and N. Sansonetto,
{\em Integrable almost-symplectic Hamiltonian systems.}
J. Math. Phys., {\bf 48} (2007), 092902, 13 pp.

\bibitem{sjamaar} 
C.K. Fok, {\it Picard group of isotropic realizations of twisted
Poisson manifolds.}  arXiv:1511.03955

\bibitem{godbillon}
C. Godbillon, {\em G\'eom\'etrie diff\'erentielle et m\'ecanique
analytique} (Hermann, Paris, 1969). 

\bibitem{GuzzoChierchiaBenettin} 
M. Guzzo, L. Chierchia and G. Benettin,
{\em The Steep Nekhoroshev's Theorem}.
ArXiv:1403.6776 (math-ph).
In press in Comm. Math. Phys. 

\bibitem{lang}
S. Lang, {\em Algebra.} Graduate Texts in Mathematics {\bf 211}
(Springer, NewYork, 2002). 

\bibitem{LibMarle}
P. Libermann and C.-M. Marle, {\em 
Symplectic geometry and analytical mechanics.}
Mathematics and its Applications {\bf 35}
(D. Reidel Publishing Co., Dordrecht, 1987).

\bibitem{nek77}
N.N. Nekhoroshev,
{\em An exponential estimate of the time of stability of
nearly integrable hamiltonian systems.}
Russian Math. Surv. {\bf 32}, 1-65 (1977).

\bibitem{poincare}
H. Poincar\'e, {\em Les m\'ethodes nouvelles de la m\'echanique
c\'eleste,} Vol. 1 (Gauthier--Villars, Paris, 1892).

\bibitem{poeschel}
J. P\"oschel, {\em Nekhoroshev estimates for quasi-convex
Hamiltonian systems.} 
Math. Z. {\bf 213} (1993), 187-216. 

\bibitem{SS1} N. Sansonetto and D. Sepe,
{\em Twisted isotropic realizations of twisted Poisson manifolds.}
J. Geom. Mech. {\bf 5} (2013), 233--256.

\bibitem{SW} 
P. \v{S}evera and A. Weinstein,
{\em Poisson geometry with a 3-form background},
in ``Noncommutative geometry and string theory''
(Yokohama, 2001), Progr. Theoret. Phys. Suppl.
{\bf 144} (2001), 145-154.

\bibitem{vaisman} I. Vaisman,
{\em Hamiltonian vector fields on almost symplectic manifolds.}
J. Math. Phys. {\bf 54} (2013), 092902, 11 pp.

\bibitem{zung1}
N.T. Zung
{\em Action-angle variables on Dirac Manifolds.}
arxiv:1204.3865

\end{thebibliography}
\end{document}